\documentclass[a4paper,reqno,11pt]{amsart}

\usepackage{amsmath}
\usepackage{amssymb}
\usepackage{cases}
\usepackage{enumitem}
\allowdisplaybreaks[4]
\usepackage{mathrsfs}
\usepackage{yfonts}
\input xy
\xyoption{all}
\usepackage{mathtools}
\DeclarePairedDelimiter{\ceil}{\lceil}{\rceil}
\DeclarePairedDelimiter{\floor}{\lfloor}{\rfloor}

\newcommand{\defeq}{\mathrel{\mathop:}=}
\newtheorem{thm}{Theorem}[section]
\newtheorem{prop}[thm]{Proposition}

\newtheorem{lem}[thm]{Lemma}

\newtheorem{cor}[thm]{Corollary}

\theoremstyle{definition}
\newtheorem{definition}[thm]{Definition}

\theoremstyle{remark}
\newtheorem{remark}[thm]{Remark}

\numberwithin{equation}{section}

\newcommand{\bQ}{\mathbb{Q}}

\newcommand\OO{{\mathcal{O}}}

\newcommand{\bR}{{\mathbb R}}

\newcommand\lct{{\rm{lct}}}   

\newcommand\Supp{{\rm{Supp}}}

\begin{document}

\title{Bounded Complements for Weak Fano Pairs with Alpha-invariants and Volumes Bounded Below}
\date{\today}

\author{Weichung Chen}
\address{Graduate School of Mathematical Sciences, the University of Tokyo, Tokyo 153-8914, Japan.}
\email{chenweic@ms.u-tokyo.ac.jp}

\begin{abstract}
We show that fixed dimensional klt weak Fano pairs with alpha-invariants and volumes bounded away from $0$ and the coefficients of the boundaries belonging to a fixed DCC set $\mathscr{S}$ form a bounded family. Moreover, such pairs admit a strong $\epsilon$-lc $\bR$-complement for some fixed $\epsilon>0$. This is an improvement of \cite{Ch18}.
\end{abstract}

\maketitle
\pagestyle{myheadings} \markboth{\hfill  Weichung Chen
\hfill}{\hfill  \hfill}


\section{Introduction}
Throughout this paper, we work over an uncountable algebraically closed field
of characteristic 0, for instance, the complex number field $\mathbb{C}$.

In the birational geometry, as the first step of moduli theory, it is interesting to consider whether a certain kind of family of varieties satisfy certain finiteness. For varieties of Fano type with bounded log discrepancies, Birkar shows in {\cite [Theorem 1.1]{Bir16b}} that
\begin{thm}\label{bab}
Fix a positive integer $d$ and a positive real number $\epsilon$. The projective varieties $X$ satisfying
\begin{enumerate}
\item $\dim X=d,$
\item there exists a boundary $B$ such that $(X, B)$ is $\epsilon$-lc, and
\item $-(K_X+B)$ is nef and big,
\end{enumerate}
form a bounded family.
\end{thm}
Theorem \ref{bab} was known as the Borisov-Alexeev-Borisov (BAB) Conjecture for decades before Birkar proved it. Equivalently, we can state Theorem \ref{bab} in the following form of boundedness of varieties of Calabi--Yau type.

\begin{thm}\label{bab2}
Fix a positive integer $d$ and a positive real number $\epsilon$. The projective varieties $X$ satisfying
\begin{enumerate}
\item $\dim X=d,$
\item there exists a boundary $B$ such that $(X, B)$ is $\epsilon$-lc, and
\item $K_X+B \sim_{\mathbb{R}}0$ and $B$ is big,
\end{enumerate}
form a bounded family.
\end{thm}

 In Theorem \ref{bab}, it is necessary to take $\epsilon>0$. In fact, klt Fano threefolds do not even form a birational family (see {\cite {Lin03}}). Nevertheless, Jiang shows in {\cite {Jia17}} that if we bound the alpha-invariants and the volumes from below, we have 

\begin{thm}[{\cite[Theorem1.6]{Jia17}}]\label{jiangbdd}
Fix a positive integer $d$ and a positive real number $\theta$. The normal projective klt Fano (i.e. $\bQ$-Fano in \cite {Jia17}) varieties $X$ satisfying
\begin{enumerate}
\item $\dim X=d$, and
\item $\alpha (X)^d(-K_X)^d>\theta$,
\end{enumerate}
form a bounded family.
\end{thm}

Inspired by Theorem \ref{jiangbdd}, it is natural to ask if certain boundedness holds for varieties of Fano type or under more general setting. Thanks to boundedness of complements by Birkar (Theorem \ref{bddcomp}), if the coefficients of the boundaries are well controlled, then the boundedness, which is one of our main theorems, holds as follows.

\begin{thm}\label{main}
Fix a positive integer $d$, positive real numbers $\theta$ and $\delta$ and a finite set $\mathscr{R}$ of rational numbers in $[0,1]$. The set of all klt Fano pairs $(X,B)$ satisfying
\begin{enumerate}
\item $\dim X=d$,
\item the coefficients of $B\in\Phi(\mathscr{R})$,
\item $\alpha(X,B)>\theta$, and
\item $(-(K_X+B))^d>\theta$
\end{enumerate}
is log bounded.
Moreover, there exists a natural number $k$, depending only on $d$, $\theta$, $\delta$ and $\mathscr{S}$ such that there exists a strong klt $k$-complement $K_X+\Theta$ of $K_X+B$.
\end{thm}

Combining Theorem \ref{main} with boundedness of complements for DCC coefficients of boundaries by Han, Liu and Shokurov, we can weaken the assumption of the coefficients of boundaries as in the following theorem.
\begin{thm}\label{2}
Fix a positive integer $d$, positive real numbers $\theta$ and $\delta$ and a DCC set $\mathscr{S}$ of real numbers in $[0,1]$. The set of all klt Fano pairs $(X,B)$ satisfying
\begin{enumerate}
\item $\dim X=d$,
\item the coefficients of $B\in\mathscr{S}$,
\item $\alpha(X,B)>\theta$, and
\item $(-(K_X+B))^d>\theta$,
\end{enumerate}
is log bounded.
Moreover, there exists a natural number $k$, finite sets $\Gamma_1=\{a_i\}_i\in[0,1]$ with $\sum_i a_i=1$ and $\Gamma_2\in[0,1]\cap\bQ$ depending only on $d$, $\theta$, $\delta$ and $\mathscr{S}$ such that there exists a strong klt $(k,\Gamma_1,\Gamma_2)$-complement $K_X+\Theta$ of $K_X+B$.
\end{thm}
\begin{remark}
Chi Li, Yuchen Liu and Chenyang Xu show the following boundedness theorem using normalized volume.
\begin{thm}[{\cite[Corollary 6.14]{LLX18}}]\label{LLX}
Fix a positive integer $d$ and a positive real number $\theta$. Then the set of all projective varieties $X$ satisfying
\begin{enumerate}
\item $\dim X=d$,
\item $(X,B)$ is a klt weak Fano pair for some effective $\mathbb{Q}$-divisor $B$ on $X$, and
\item $\alpha(X,B)^d(-(K_X+B))^d>\theta$,
\end{enumerate}
is bounded.
\end{thm}
\end{remark}
\medskip

\section{Preliminaries}
We adopt the standard notation and definitions in \cite{KMM} 
and \cite{KM}, and will freely use them.

\subsection{Pairs and singularities}
A {\it sub-pair} $(X, B)$ consists of a normal projective variety $X$ 
and an $\bR$-divisor $B$ on $X$ such that $K_X+B$ is $\bR$-Cartier. $B$ is called the {\it sub-boundary} of this pair.

A {\it log pair} $(X, B)$ is a sub-pair with $B\geq 0$. We call $B$ a {\it boundary} in this case.

Let $f\colon Y\rightarrow X$ be a log
resolution of the log pair $(X, B)$, write
\[
K_Y =f^*(K_X+B)+\sum a_iF_i,
\]
where $F_i$ are distinct prime divisors.  
For a non-negative real number $\epsilon$, the log pair $(X,B)$ is called
\begin{itemize}
\item[(a)] \emph{$\epsilon$-kawamata log terminal} (\emph{$\epsilon$-klt},
for short) if $a_i> -1+\epsilon$ for all $i$;
\item[(b)] \emph{$\epsilon$-log canonical} (\emph{$\epsilon$-lc}, for
short) if $a_i\geq  -1+\epsilon$ for all $i$;
\end{itemize}

Usually we write $X$ instead of $(X,0)$ in the case when $B=0$.
Note that $0$-klt (resp. $0$-lc) is just klt (resp. lc) in the usual sense. Also note that 
$\epsilon$-lc singularities only make sense if $\epsilon\in [0,1]$, and  $\epsilon$-klt 
singularities only make sense if $\epsilon\in [0,1)$.

Similarly, sub-$\epsilon$-klt and sub-$\epsilon$-lc sub-pairs can be defined.

The {\it log discrepancy} of the divisor $F_i$ is defined to be $a(F_i, X, B)=1+a_i$.
It does not depend on the choice of the log resolution $f$.

$F_i$ is called a {\it non-klt place} of $(X, B)$  if $a_i\leq -1$.
A subvariety $V\subset X$ is called a {\it non-klt center} of 
$(X, B)$ if it is the image of a non-klt place. 
The {\it non-klt locus} $\text{Nklt}(X, B)$ is the union of 
all non-klt centers of $(X, B)$.
We recall the Koll\'{a}r-Shokurov connectedness lemma.
\begin{lem}[cf. \cite{Sho93}, \cite{Sho94} and {\cite[Theorem 17.4]{Kol92}}]\label{cnnlem}
Let $(X,B)$ be a log pair, and let $\pi$: $X\rightarrow S$ be a proper morphism with connected fibers. Suppose $-(K_X+B)$ is $\pi$-nef and $\pi$-big. Then $\mathrm{Nklt}(X, B)\cap X_s$ is connected for any fiber $X_s$ of $\pi$.
\end{lem}

\subsection{Fano pairs and Calabi--Yau pairs}
A projective pair $(X,B)$ is a {\it Fano} (resp. {\it weak Fano}, resp. {\it Calabi--Yau}) pair if it is lc and $-(K_X+B)$ is ample (resp. $-(K_X+B)$ is nef and big, resp. $K_X+B\equiv 0$).
A projective variety $X$ is called Fano, (resp. Calabi--Yau) if $(X,0)$ is Fano (resp. Calabi--Yau). It is called {\it $\mathbb{Q}$-Fano} if it is klt and Fano. It is called {\it of Fano type} if $(X,B)$ is klt weak Fano for some boundary $B$.

\subsection{Bounded pairs}\label{sec.bdd}
A collection of varieties $ \mathcal{D}$ is
said to be \emph{bounded} (resp. 
\emph{birationally bounded})
if there exists 
$h\colon \mathcal{Z}\rightarrow S$ a projective morphism 
of schemes of finite type such that
each $X\in \mathcal{D}$ is isomorphic (resp. birational) to $\mathcal{Z}_s$ 
for some closed point $s\in S$.

A couple $(X,D)$ consists of a normal projective variety $X$ and a reduced divisor $D$ on X. Note that we do not require $K_X+D$ to be $\bQ$-Cartier here.

We say that a collection of couples $\mathcal{D}$ is 
{\it log birationally bounded} (resp.  \emph{log bounded})
if there is a  quasi-projective scheme $\mathcal{Z}$, a 
reduced divisor $\mathcal{E}$ on $\mathcal Z$, and a 
projective morphism $h\colon \mathcal{Z}\to S$, where 
$S$ is of finite type and $\mathcal{E}$ does not contain 
any fiber, such that for every $(X,D)\in \mathcal{D}$, 
there is a closed point $s \in S$ and a birational (resp. isomorphic)
map $f \colon \mathcal{Z}_s \dashrightarrow X$ such that $\mathcal{E}_s$ contains the support of $f_*^{-1}D$ 
and any $f$-exceptional divisor.

A set of log pairs $\mathcal{P}$ is 
{\it log birationally bounded} (resp.  \emph{log bounded})
if the set of the corresponding couples $\{(X,\mathrm{Supp}B)|(X,B)\in\mathcal{P}\}$ is log birationally bounded (resp. log bounded).

\subsection{Volumes}
Let $X$ be a $d$-dimensional normal projective variety  and $D$ 
a Cartier divisor on $X$. The {\it volume} of $D$ is the real number
\[
\mathrm{vol}(X, D)=\limsup_{m\rightarrow \infty}\frac{h^0(X,\OO_X(mD))}{m^d/d!}.
\]
For more backgrounds on the volume, see \cite[2.2.C]{Positivity1}. 
By the homogenous property and  continuity of the volume, we 
can extend the definition to $\bR$-Cartier $\bR$-divisors. 
Moreover, if $D$ is a nef $\bR$-divisor, then vol$(X, D)=D^d$.

\subsection{Complements}
\begin{definition}
Let $(X,B)$ be a pair and $n$ a positive integer. We write $B=\floor{B}+\{B\}$. An {\it n-complement} of $K_X+B$ is a divisor of the form $K_X+B^+$ such that 
\begin{enumerate}
\item $(X,B^+)$ is lc,
\item $n(K_X+B^+)\sim 0$, and
\item $nB^+\geq n\floor{B}+\floor{(n+1)\{B\}}$.
\end{enumerate}
If moreover, $B^+\geq B$ (resp. $(X,B^+)$ is klt, resp. $(X,B^+)$ is $\epsilon$-lc, where $\epsilon>0$), we say that the complement is {\it strong} (resp. klt, resp. $\epsilon$-lc).

We say that $(X,B)$ has an {\it$\bR$-complement} or is {\it complementary} if there exists $\overline{B}\geq B$ such that $(X,\overline{B})$ is lc.
In this case, we call $K_X+\overline{B}$ an $\bR$-complement of $K_X+B$.

\end{definition}
\begin{definition}
For a subset $\mathscr{R}$ of $[0,1]$, we define the set of {\it hyperstandard multiplicities} associated to $\mathscr{R}$ to be
\[\Phi(\mathscr{R})=\{ 1-\frac{r}{m}| r\in\mathscr{R},\ m\in\mathbb{N} \}.\]
\end{definition}
Note that the only possible accumulating point of $\Phi(\mathscr{R})$  is $1$ if $\mathscr{R}$ is finite.
Birkar shows the following boundedness of complements.
\begin{thm}[{\cite [Theorem 1.7]{Bir16a}}]\label{bddcomp}
Fix a positive integer $d$ and a finite set $\mathscr{R}$ of rational numbers in $[0,1]$. Then there exists a positive integer $n$ depending only on $d$ and $\mathscr{R}$, such that if  $(X,B)$ is a projective pair with
\begin{enumerate}
\item $(X,B)$ is lc dimension $d$,
\item the coefficients of $B\in\Phi(\mathscr{R})$,
\item $X$ is of Fano type, and
\item $-(K_X+B)$ is nef,
\end{enumerate}
then there is an $n$-complement $K_X+B^+$ of $K_X+B$ such  that $B^+\geq B$.
\end{thm}

Based on Theorem \ref{bddcomp} and ACC for log canonical threshold [{\cite [Theorem 1.1]{HMX14}}],  Filipazzi and Moraga showed the following existence of bounded complements for DCC coefficients of the boundaries.
\begin{thm}[{\cite [Theorem 1.2]{FM18}}]\label{FM}
Fix a positive integer $d$ and a closed DCC set $\mathscr{S}$ of rational numbers in $[0,1]$. Then there exists a positive integer $n$ depending only on $d$ and $\mathscr{S}$, such that if  $(X,B)$ is a projective pair with
\begin{enumerate}
\item $(X,B)$ is lc dimension $d$,
\item the coefficients of $B\in\mathscr{S}$,
\item $X$ is of Fano type, and
\item $-(K_X+B)$ is nef,
\end{enumerate}
then there is an $n$-complement $K_X+B^+$ of $K_X+B$ such  that $B^+\geq B$.
\end{thm}

\begin{thm}[{\cite[Theorem 5.22]{HLS}}]\label{ACCPELC}
Let $d$ be a natural number and $\Gamma\subset [0,1]$ be a closed DCC set. Then there exists a finite subset  $\Gamma_1$ of $\Gamma$ and a projection $g:\Gamma\rightarrow \Gamma_1$ such that $g(g(x))=g(x)$, $g(x)\geq x$ and $g(y)\geq g(x)$ for every $x\leq y\in \Gamma$, depending only on $d$ and $\Gamma$ satisfying the following. Suppose $X,B\defeq \sum b_i B_i$ is lc of dimension $d$ where $b_i\in\Gamma$, $B_i\geq 0$ is $\bQ$-Cartier Weil divisor for any $i$, $(X,B)$ has an $\bR$-complement and $X$ is of Fano type, then
\begin{enumerate}
\item $(X,\sum g(b_i)B_i)$ is lc, and
\item $-(K_X+\sum g(b_i)B_i)$ is pseudo-effective.
\end{enumerate}
\end{thm}

We recall the construction of $(n,\Gamma_1,\Gamma_2)$-complement by Han, Liu and Shokurov.
\begin{thm}[{\cite[Theorem 1.13]{HLS}}]\label{HLS}
Let $d$ be a natural number $\delta>0$ be a positive real number and $\Gamma\subset [0,1]$ be a closed DCC set. Then there exist integers $n>0$ and $r>0$, finite sets $\Gamma_1\subset [0,1]$,
$\{ a_i\}_{i=1}^r\in (0,1]$,
$\Gamma_2=\cup_{i=1}^r \Gamma'_i \subset\bQ\cap [0,1]$, a projection $g:\Gamma\rightarrow \Gamma_1$ and bijections $g_i:\Gamma_1\rightarrow\Gamma'_i$ for every $1\leq i\leq r$ which only depend only on $d$ and $\Gamma$ satisfying the following. 

$\Gamma_1$ and $g$ are given by Theorem \ref{ACCPELC}. $g(x)=\sum_{i=1}^r a_i (g_i\circ g(x))$ and $\sum_{i=1}^r a_i=1$. $|g_i(x)-x|<\delta$ for every $1\leq i\leq r$ and $x\in \Gamma_1$.
Assume $X$ is a normal variety and
$B$ is an effective divisor on $X$ with the following conditions:
\begin{enumerate}
\item X is of Fano type,
\item $\dim X\leq d$, 
\item the coefficients of $B\in \Gamma$,
\item $(X,B)$ has an $\bR$-complement, and
\item we can write $B\defeq \sum b_j B_j$ where $b_j\in\Gamma$ and $B_j\geq 0$ is a $\bQ$-Cartier Weil divisor for any $j$,
\end{enumerate}
then $(X,\sum g(b_j)B_j)$ and $(X,\sum_j g_i\circ g(b_j)B_j)$ are lc for every and $K_X+\sum_j g_i\circ g(b_j)B_j$ has a strong $n$-complement for every $1\leq i\leq r$. In particular, $(X,\sum g(b_j)B_j)$ has an $\bR$-complement.
\end{thm}

\begin{definition}[{\cite[Definition 1.12]{HLS}}]\label{HLScomplement}
Assume $(X,B)$ is a normal pair,
$n$ is a positive integer,
and $\Gamma_1\subset [0,1]$
and $\Gamma_2\subset\bQ\cap [0,1]$ are two finite sets.
We say that $K_X+B^+$ is an $(n,\Gamma_1,\Gamma_2)$-complement of $(X,B)$ if the following conditions are satisfied.
\begin{enumerate}
\item $B^+\geq B$,
\item there is a finite set $\{a_i\}_{i=1}^r\subset\Gamma_1$ with $\sum a_i=1$,
\item there are divisors $B_i\geq 0$ on $X$ with coefficients of $B_i\in \Gamma_2$ for every $1\leq i\leq r$,
\item $\sum a_iB_i=B^+$, and
\item $K_X+B_i$ is an $n$-complement of itself for each $1\leq i\leq r$.
\end{enumerate}
\end{definition}

\begin{definition}
Under the notation of Definition \ref{HLScomplement}, if $(X,B_i)$ are klt, then we say that the $(n,\Gamma_1,\Gamma_2)$-complement $K_X+B^+$ is klt. Note that this implies that $(X,B_i)$ and $(X,B^+)$ are $\frac{1}{n}$-lc.
\end{definition}

\begin{thm}\label{g'}
Under the notation of  Theorem \ref{HLS}, for each $1\leq i\leq r$ there exist a closed DCC set of rational numbers $\mathscr{S}_i$ and a function $g'_i:\Gamma\rightarrow \mathscr{S}_i$ such that $B\leq\sum_i(a_i\sum_j g'_i(b_j)B_j)$ and $\sum_j g'_i(b_j)B_j\leq\frac{B+3\sum_j g(b_j)B_j}{4}$ for every $1\leq i\leq r$.
\end{thm}
\begin{proof}
By shrinking $\Gamma_1$, we may assume that it is the image of $g$ and so $g$ is identity on $\Gamma_1$. We may also assume that $0$ is not in $\Gamma_1$. We define a finite rational set $\Gamma^-$ and a function $g^-:\Gamma\rightarrow \Gamma^-$ as following. For $\gamma\in\Gamma_1$, let $\gamma^-$ be any rational number in the intevial \[(\max\{0,x|x\in\Gamma,g(x)<\gamma\},\min\{x|x\in\Gamma,g(x)\geq \gamma\}).\] Note that here the maximum and minimum exist and $\max\{0,x|x\in\Gamma,g(x)<\gamma\}<\min\{x|x\in\Gamma,g(x)\geq \gamma\}$. Let $g^-(x)=g(x)^-$ for every $x\in\Gamma$.
It follows by construction that $g^-(x)\leq x<g(x)$ for every $x\in\Gamma$.
By abuse of notations, we denote the functions $g_i\circ g$ as $g_i$. Since $\Gamma_1$ and $\Gamma^-$ are finite, by possibly replacing $\delta$, we may assume that $g(x)-g_i(x)\leq \frac{g(x)-g^-(x)}{2}$ for every $x\in\Gamma$. Therefore we have $2(g_i(x)-g^-(x))\geq(g(x)-g^-(x))$ for every $x\in\Gamma$. Let $f:\bQ\rightarrow \mathbb{N}$ be any bijection. We now define the functions $g'_i$ as following.
Assume that $x\in\Gamma$, let $x'=g^-(x)+\frac{(g_i(x)-g^-(x))(x-g^-(x))}{(g(x)-g^-(x))}$. Then
\begin{align*}
g_i(x)-x'=&g_i(x)-g^-(x)-\frac{(g_i(x)-g^-(x))(x-g^-(x))}{(g(x)-g^-(x))}\\
=&(g_i(x)-g^-(x))\frac{(g(x)-g^-(x))-(x-g^-(x))}{(g(x)-g^-(x))}\\
=&(g_i(x)-g^-(x))\frac{(g(x)-x)}{(g(x)-g^-(x))}\\
\geq&0,
\end{align*}
where the equality holds if and only if $g(x)=x$.
Let $g'_i(x)$ be the rational number in $[x',\frac{x'+g_i(x)}{2}]$ with the smallest $f$ value. Note that here, if $x'=\frac{x'+g_i(x)}{2}$, then it is a rational number in $\Gamma_i$.
Let $\mathscr{S}_i$ be the closure or the image of $g'_i$. 

We now argue that $\mathscr{S}_i$ is a rational DCC set. 
It is enough to show that the image of $g'_i$ is a DCC set with rational accumulation points.

Fix $i$ for each $1\leq i\leq r$. Let $\{x_l\}_l$ be a sequence of elements in $\Gamma$ such that $\{g'_i(x_l)\}_l$ is and strictly converges to some point $x_0\in\mathscr{S}_i$.
Suppose either $\{g'_i(x_l)\}_l$ is strictly decreasing or $x_0$ is irrational.
By construction, $g'_i(x_l)$ is the element with the smallest $f$ value in $\bQ\cap[x_l',\frac{x_l'+g_i(x_l)}{2}]$, where $x_l'=g^-(x_l)+\frac{(g_i(x_l)-g^-(x_l))(x_l-g^-(x_l))}{(g^+(x_l)-g^-(x_l))}$.
Because the image of $g^-$, $g_i$ and $g^+$ are finite, $\frac{(g_i(x)-g^-(x))}{(g^+(x)-g^-(x))}>0$ for every $x\in\Gamma$, and $\Gamma$ is DCC, it follows that $\{x_l'\}_l$ and $\{\frac{x_l'+g_i(x_l)}{2}\}_l$ are both DCC. Passing through a subsequence, we may assume that $\{x_l'\}_l$ and $\{\frac{x_l'+g_i(x_l)}{2}\}_l$ are non-decreasing.
Passing to a subsequence, we may assume the sequence $\{g_i(x_l)\}_l$ is a constant sequence.
Further, if $\{x_l'\}_l$ is a constant sequence, $\{\frac{x_l'+g_i(x_l)}{2}\}_l$ is also a constant sequence, and this implies $\{g'_i(x_l)\}_l$ is also constant, which is a contradiction.
Now we may assume that $\{x_l'\}_l$ and $\{\frac{x_l'+g_i(x_l)}{2}\}_l$ are strictly increasing and $\{x_l'\}_l$ converges to $x'_0\leq g_i(x_l)\in \Gamma'_i$.
If $x'_0=g_i(x_l)$, then neither $\{g'_i(x_l)\}_l$ is strictly decreasing nor $x_0=g_i(x_l)$ is irrational, which is a contradiction.
So we have $x'_0<g_i(x_l)$. Let $y$ be a rational number in $[x'_0,\frac{x'_0+g_i(x_l)}{2})$
Passing through a subsequence again, we may assume that $\frac{x'+g_i(x)}{2}>y$ for every $l$.
This implies that $x_l'<x'_0<y<\frac{x_l'+g_i(x_l)}{2}$ for every $l$. 
Then $f(g'_i(x_l))\leq f(y)$ for every $l$ because $g'_i(x_l)$ is the rational number in $[x'_l,\frac{x'+g_i(x)}{2}]$ with the smallest $f$ value.
But the set $\{f(g'_i(x_l))\}$ is infinite because  $\{g'_i(x_l)\}_l$ is strictly converging. This contradicts to our assumption that $f$ is a bijection.

It then remains to show that
$B\leq\sum_i(a_i\sum_j g'_i(b_j)B_j)$ and $\sum_j g'_i(b_j)B_j\leq\frac{B+3\sum_j g(b_j)B_j}{4}$ for every $1\leq i\leq r$.
We simply compute that
\begin{align*}
\sum_i(a_i\sum_j g'_i(b_j)B_j)\geq&\sum_i(a_i\sum_j (g^-(b_j)+\frac{(g_i(b_j)-g^-(b_j))(b_j-g^-(b_j))}{(g(b_j)-g^-(b_j))})B_j)\\
=&\sum_j (g^-(b_j)+\frac{(\sum_i(a_ig_i(b_j))-g^-(b_j))(b_j-g^-(b_j))}{(g(b_j)-g^-(b_j))})B_j\\
=&\sum_j (g^-(b_j)+\frac{(g(b_j)-g^-(b_j))(b_j-g^-(b_j))}{(g(b_j)-g^-(b_j))})B_j\\
=&\sum_jb_jB_j=B.
\end{align*}
On the other hand, we have
\begin{align*}
\sum_j g(b_j)B_j-\sum_j g'_i(b_j)B_j)&=\sum_i(a_i\sum_j g_i(b_j)B_j)-\sum_i(a_i\sum_j g'_i(b_j)B_j)\\
=&\sum_j\sum_ia_i (g_i(b_j)-g'_i(b_j))B_j.\\
\end{align*}
We compute that for each $i$,
\begin{align*}
g_i(b_j)-g'_i(b_j)&\geq g_i(b_j)-\frac{g_i(b_j)+g^-(b_j)+\frac{(g_i(b_j)-g^-(b_j))(b_j-g^-(b_j))}{(g(b_j)-g^-(b_j))}}{2}\\
=&\frac{g_i(b_j)-g^-(b_j)-\frac{(g_i(b_j)-g^-(b_j))(b_j-g^-(b_j))}{(g(b_j)-g^-(b_j))}}{2}\\
=&\frac{1}{2}(g_i(b_j)-g^-(b_j))\frac{(g(b_j)-g^-(b_j))-(b_j-g^-(b_j))}{(g(b_j)-g^-(b_j))}\\
=&\frac{1}{2}(g_i(b_j)-g^-(b_j))\frac{(g(b_j)-b_j)}{(g(b_j)-g^-(b_j))}\\
\geq&\frac{(g(b_j)-b_j)}{4}.
\end{align*}
It follows that $\sum_j g(b_j)B_j-\sum_j g'_i(b_j)B_j)\geq \frac{\sum_jg(b_j)B_j-B}{4}$. Therefore, we have 
\begin{align*}
\sum_j g'_i(b_j)B_j\leq&\sum_jg(b_j)B_j-\frac{\sum_jg(b_j)B_j-B}{4}\\
=&\frac{B+3\sum_j g(b_j)B_j}{4}.
\end{align*}
\end{proof}

\subsection{$\alpha$-invariants and log canonical thresholds}
\begin{definition}
Let $(X,B)$ be a projective lc pair and let $D$ be an effective ${\bR}$-Cartier divisor, we define the {\it log canonical threshold} of $D$ with respect of $(X,B)$ to be
\[\lct((X,B),D)=\sup\{\text{$t\in \mathbb{R}$ $|$ $(X,B+tD)$ is lc}\}.\]
The log canonical threshold of $|D|_{\bR}$ with respect of $(X,B)$ is defined to be
\[\lct((X,B),|D|_{\bR})=\inf\{ \mathrm{lct}((X,B),M)| M\in|D|_{\bR}\}.\]
\end{definition}
\begin{definition}
Let $(X,B)$ be a projective lc pair such that $|-(K_X+B)|_{\bR}$ is non-empty, we define the {\it $\alpha$-invariant} of $(X,B)$ to be
\[\alpha(X,B)=\mathrm{lct}((X,B),|-(K_X+B)|_{\bR}).\]
In the case when $B=0$, we usually write $\alpha(X)\defeq \alpha(X,0)$ for convenience.
\end{definition}
Now we consider $\alpha(X,B)^{d}(-(K_X+B))^d$ as an invariant for $d$-dimensional klt Fano pairs $(X,B)$. It is well known that this invariant has an upper bound, which can be given by the following lemma.
\begin{lem}[{\cite [Theorem 6.7.1]{Kol97}}]\label{kollar}
Let $(X,B)$ be a klt pair of dimension $d$. Then we have
$$\mathrm{lct}((X,B),|H|_{\bR})^dH^d\leq d^d$$
for any nef and big $\bQ$-Cartier divisor $H$ on $X$.
\end{lem}

\subsection{Potentially birational divisors}
\begin{definition}[cf. {\cite[Difinition 3.5.3]{HMX14}}]
Let $X$ be a projective normal variety, and $D$ a big $\mathbb{Q}$-Cartier $\bQ$-divisor on $X$. Then, we say that $D$ is {\it potentially birational} if for any two general points $x$ and $y$ of $X$, there is an effective $\bQ$-divisor $\Delta\sim_{\bQ}(1-\epsilon)D$ for some $0<\epsilon<1$, such that, after possibly switching $x$ and $y$, $(X,\Delta)$ is not klt at $y$, lc at $x$ and $x$ is a non-klt center.
\end{definition}

\subsection{Descending chain condition}
\begin{definition}
A set of real numbers $\mathscr{S}$ is said to {\it satisfy descending chain condition (DCC for short)} if for every non-empty subset $S$ of $\mathscr{S}$, there is a minimum element in $S$. $\mathscr{S}$ is called a DCC set if it satisfies DCC.
\end{definition}

\section{Proofs of Theorems}\label{sec 3}
Now we restate and prove the theorems in Section 1.
We recall the following key Lemma of \cite{Ch18}.
\begin{thm}[{\cite[Theorem 3.1]{Ch18}}]\label{thm2}
Fix a positive integer $d$ and a positive real number $\theta$. Then there is a number $m$ depending only on $d$ and $\theta$ such that if $X$ is a projective normal variety satisfying
\begin{enumerate}
\item $\dim X=d,$
\item there exists a boundary $B$ such that $(X, B)$ is klt,
\item there is a nef $\bQ$-Cartier $\bQ$-divisor $H$ on $X$ with $\mathrm{lct}((X,B),|H|_{\bR})>\theta$, and
\item $H^d>\theta$,
\end{enumerate}
then $|K_X+\ceil{mH}|$ defines a birational map and $mH$ is potentially birational.
\end{thm}
\begin{cor}[{\cite[Corollary 3.2]{Ch18}}]\label{cor2}
Fix positive integers $d$, $n$, and a positive real number $\theta$. Then there is a number $m$ depending only on $d$, $n$ and $\theta$ such that if $(X,B)$ is a weak Fano pair satisfying
\begin{enumerate}
\item $\dim X=d,$
\item $\alpha(X,B)>\theta$,
\item $K_X+B$ is a $\bQ$-Cartier $\bQ$-divisor,
\item $(-K_X-B)^d>\theta$, and
\item there is an $n$-complement $K_X+B^+$ of $K_X+B$ with $B^+\geq B$,
\end{enumerate}
then $|\floor{m(B^+-B)}|$ defines a birational map.
\end{cor}

We recall the following theorem by Hacon and Xu.
\begin{thm}[{\cite [Theorem 1.3]{HX15}}]\label{hx}
Fix a positive integer $d$ and a DCC set $\mathscr{I}$ of rational numbers in $[0,1]$. The set of all projective pairs $(X,B)$ satisfying
\begin{enumerate}
\item $(X,B)$ is klt log Calabi--Yau of dimension $d$,
\item $B$ is big, and
\item the coefficients of $B\in\mathscr{I}$,
\end{enumerate}
forms a bounded family.
\end{thm}

Then we recall the following log-version of {\cite [Lemma 2.26]{Bir16a}}.
\begin{lem}\label{qfact}
Fix positive integers $d$, $k$ and a non-negative real number $\epsilon$. Let  $\mathcal{P}$ be a set of klt weak Fano pairs of dimension $d$. Assume that for every element $(Y,B_Y)\in\mathcal{P}$, there is a $k$-complement (resp. $\bR$-complement) $K_Y+B^+_Y$ of $K_Y+B_Y$ such that $(Y,B^+_Y)$ is $\epsilon$-lc and $B^+_Y\geq B_Y$. Let  $\mathcal{Q}$ be the set of projective pairs $(X,B)$ such that
\begin{enumerate}
\item there is $(Y,B_Y)\in\mathcal{P}$ and a birational map $X\dashrightarrow Y$,
\item there is a common resolution $\phi :W\rightarrow Y$ and $\psi:W\rightarrow X$, and
\item $\phi^*(K_Y+B_Y)\geq\psi^*(K_X+B)$.
\end{enumerate}
Then for every element $(X,B)\in\mathcal{Q}$, there is a $k$-complement $K_X+B^+$ of $K_X+B$ such that $(X,B^+)$ is $\epsilon$-lc and $B^+\geq B$.
\end{lem}
\begin{proof}
Let $K_X+B^+$ be the crepant pullback of $K_Y+B_Y^+$ to $X$. Then $(X,B^+)$ is $\epsilon$-lc. Since $k(K_Y+B^+_Y)\sim 0$ (resp. $K_Y+B^+_Y\sim_{\bR}0$) and $\phi^*(K_Y+B_Y)\geq\psi^*(K_X+B)$, $B^+-B\geq\psi_*\phi^*(B_Y^+-B_Y)\geq 0$ and $K_X+B^+$  is a strong $k$-complement (resp. $\bR$-complement) of $K_X+B$.
\end{proof}

\begin{lem}\label{alpineq}
Let $(X,B)$ be a projective lc weak Fano pair with $\alpha(X,B)\leq 1$. Suppose $\phi:X\dashrightarrow Y$ is a partial MMP (of any divisor). Let $B_Y=\phi_*(B)$. Then $(Y,B_Y)$ is lc and $\alpha(Y,B_Y)\geq\alpha(X,B)$. If moreover $(X,B)$ is klt, then $(Y,B_Y)$ is also klt.
\end{lem}
\begin{proof}
Since (X,B) is weak Fano, there is an $\bR$-complement $K_X+\overline{B}$ of $K_X+B$.
Let $\overline{B}_Y=\phi_*(\overline{B})$, then $K_Y+\overline{B}_Y$ is an $\bR$-complement of $K_Y+B_Y$. In particular, $(Y,B_Y)$ is lc and $|-(K_Y+B_Y)|_{\bR}$ is non-empty.
Let $t=\alpha(X,B)$. 
Now we consider the following argument for any $M_Y\in |-(K_Y+B_Y)|_{\bR}$. Let $K_X+B+M$ be the crepant pullback of $K_Y+B_Y+M_Y$ to $X$. As in the proof of Lemma \ref{qfact}, we have $M\in |-(K_X+B)|_{\bR}$ and $\phi_*(M)=M_Y$.
By assumption, $(X,B+tM)$ is lc and $-(K_X+B+tM)\sim_{\bR} -(1-t)(K_X+B)$ is nef and big if $t\neq 1$. So there is an effective divisor $G$ on $X$ such that $(X,B+tM+G)$ is lc Calabi--Yau. Let $G_Y=\phi_*(G)$. Then $(Y,B_Y+tM_Y+G_Y)$ is an $\bR$-complement of $K_Y+B_Y+tM_Y$. In particular, $(Y,B_Y+tM_Y)$ is lc. Therefore, we have $\alpha(Y,B_Y)\geq t$. If moreover $(X,B)$ is klt, then we may assume $(X,\overline{B})$ is klt and therefore $(Y,B_Y)$ is klt.
\end{proof}

Next, we recall the following proposition of Birkar with a small observation.
\begin{prop}[{\cite[Proposition 4.4]{Bir16a}}]\label{birkar2}
Fix positive integers $d$, $v$ and a positive real number $\epsilon$. Then there exists a bounded set of couples $\mathcal{P}$ and a positive real number $c$ depending only on $d$, $v$ and $\epsilon$ satisfies the following. Assume 
\begin{itemize}
	\item $X$ is a normal projective variety of dimension $d$,
	\item $B$ is an effective $\bR$-divisor with coefficient at least $\epsilon$,
	\item $M$ is a $\bQ$-divisor with $|M|\defeq|\floor{M}|$ defining a birational map,
	\item $M-(K_X+B)$ is pseudo-effective,
	\item $\mathrm{vol}(M)<v$, and
	\item $\mu_D(B+M)\geq 1$ for every component $D$ of $M$.
\end{itemize}
Then there is a projective log smooth couple $(\overline{W}, \Sigma_{\overline{W}})\in \mathcal{P}$, a birational map $\overline{W}\dashrightarrow X$ and a common resolution $X'$  of this map such that
\begin{enumerate}
\item Supp$\Sigma_{\overline{W}}$ contains the exceptional divisor of $\overline{W}\dashrightarrow X$ and the birational transform of Supp$(B+M)$, and
\item there is a resolution $\phi : W \rightarrow X$ such that $M_W\defeq \phi^*M\sim A_W+R_W$ where $A_W$ is the movable part of $|M_W|$, $|A_W|$ is base point free, $\psi:X'\rightarrow X$ factors through $W$ and $A_{X'}\sim 0/\overline{W}$, where $A_{X'}$ is the pullback of $A_W$ on $X'$.
\end{enumerate}
Moreover, if M is nef and $M_{\overline{W}}$ is the pushdown of $M_{X'}\defeq \psi^*M$. Then each coefficient of $M_{\overline{W}}$ is at most $c$.
\end{prop}
Note that in the original statement of {\cite[Proposition 4.4]{Bir16a}}, $M$ is assumed to be nef. We observe from Birkar's proof that the nefness of $M$ is used only when showing the existence of $c$ and is not necessary when showing (1) and (2) of proposition \ref{birkar2}.

Now we are ready to show the main theorem of this paper.
The idea is to follow the strategy of {\cite[Proposition 7.13]{Bir16a}}, which is to construct a klt complement with coefficients in a finite set depending only on $d$, $\theta$ and $\mathscr{R}$, and then apply Theorem \ref{hx}.


\begin{thm}\label{thm4.2}
Fix a positive integer $d$, a positive real number $\theta$ and a closed DCC set $\mathscr{S}$ of rational numbers in $[0,1]$. The set $\mathcal{D}$ of all klt weak Fano pairs $(X,B)$ satisfying
\begin{enumerate}
\item $\dim X=d$,
\item the coefficients of $B\in\mathscr{S}$,
\item $\alpha(X,B)>\theta$, and
\item $(-(K_X+B))^d>\theta$
\end{enumerate}
forms a log bounded family.
Moreover, there is a rational number $k$ depending only on $d$, $\theta$ and $\mathscr{S}$, such that every element $(X,B)\in\mathcal{D}$, there is a strong klt $k$-complement $K_X+\Theta$ of $K_X+B$.
\end{thm}
\begin{proof}
By Theorem \ref{hx}, it is enough to show the existence of $k$.

By Lemma \ref{qfact}, replacing by a small $\bQ$-factorialisation of $X$, we may assume $X$ is $\bQ$-factorial.

By Theorem \ref{FM}, there is an $n$-complement $K_X+B^+$ of $K_X+B$ such that $B^+\geq B$.

Let $m$ be given by Corollary \ref{cor2} such that  $|\floor{m(B^+-B)}|$ defines a birational map. Replacing $n$ and $m$ by $2mn$, we may assume $n=m>1$. On the other hand, by Lemma \ref{kollar}, vol$(\floor{m(B^+-B)})\leq$vol$(m(B^+-B))<v$ for some $v$ depending only on $d$, $m$ and $\theta$.

Let $M$ be a general element of $|\floor{m(B^+-B)}|$. By Proposition \ref{birkar2}, there is a bounded set of couples $\mathcal{P}$ depending only on $d$, $m$ and $\theta$, such that there is a projective log smooth couple $(\overline{W}, \Sigma_{\overline{W}})\in \mathcal{P}$, a birational map $\overline{W}\dashrightarrow X$ and a common resolution $X'$  of this map such that
\begin{enumerate}
\item Supp$\Sigma_{\overline{W}}$ contains the exceptional divisor of $\overline{W}\dashrightarrow X$ and the birational transform of Supp$(B^++M)$, and
\item there is a resolution $\phi : W \rightarrow X$ such that $M_W\defeq \phi^*M\sim A_W+R_W$ where $A_W$ is the movable part of $|M_W|$, $|A_W|$ is base point free, $X'\rightarrow X$ factors through $W$ and $A_{X'}\sim 0/\overline{W}$, where $A_{X'}$ is the pullback of $A_W$ on $X'$.
\end{enumerate}
Since $M$ is a general element of $|\floor{m(B^+-B)}|$, we may assume $M_{W}=A_{W}+R_W$ and $A_{W}$ is general in $|A_{W}|$. 
In particular, if $A_{\overline{W}}$ is the pushdown of $A_{W}|_{X'}$ to $\overline{W}$, then $A_{\overline{W}}\leq\Sigma_{\overline{W}}$.
Let $M$, $A$, $R$ be the pushdowns of $M_W$, $A_W$, $R_W$ to $X$.



Since $|A_{\overline{W}}|$ defines a birational contraction and $A_{\overline{W}}\leq\Sigma_{\overline{W}}$, there exists $l\in \mathbb{N}$ depending only on $\mathcal{P}$ such that $lA_{\overline{W}}\sim G_{\overline{W}}$ for some $G_{\overline{W}}\geq 0$ whose support contains $\Sigma_{\overline{W}}$. Let $K_{\overline{W}}+B^+_{\overline{W}}$ be the crepant pullback of $K_X+B^+$ to ${\overline{W}}$. Then $(\overline{W},B^+_{\overline{W}})$ is sub-lc and
$$\text{Supp}B^+_{\overline{W}}\subseteq\text{Supp}\Sigma_{\overline{W}}\subseteq \text{Supp}G_{\overline{W}}.$$
Let $G$ be the pushdown of $G_{X'}\defeq G_{\overline{W}}|_{X'}$ to $X$. 
Since $A_{X'}$ is the pullback of $A_{\overline{W}}$, we obtain $lA_{X'}\sim G_{X'}$ and $lA\sim G$. Therefore, $G+lR+l\{m(B^+-B)\}\sim_{\bQ}m(B^+-B)$.

Take a positive rational number $t\leq (lm)^{-d}\theta$, then \[(X,B+t(G+lR+l\{m(B^+-B)\}))\] is klt. Moreover, we have
\[-K_X-B-t(G+lR+l\{m(B^+-B)\})\sim_{\bQ}B^+-B-t(lm(B^+-B)).\]
By replacing $t$, we may assume $t<\frac{1}{lm}$.
Since 
\[B^+-B-t(lm(B^+-B))=(1-tlm)(B^+-B)\geq 0,\]
$B^+-B-t(lm(B^+-B))$ is nef and big. Therefore \[(X,B+t(G+lR+l\{m(B^+-B)\}))\] is klt weak Fano.

Now we argue that the coefficients of $B+t(G+lR+l\{m(B^+-B)\})$ are in a closed DCC set of rational numbers in $[0,1]$. It is clear that the coefficients of $B+t(G+lR+l\{m(B^+-B)\})$ are in a set of rational numbers in $[0,1]$, so it is enough to show that they are in a set with rational accumulation points.
Write \[B+tl\{m(B^+-B)\}=(1-tlm)B+tlmB^+-tl\floor{m(B^+-B)}.\]
Since $m$, $l$ and $t$ are determined by the fixed terms $d$, $\theta$ and $\mathscr{R}$, the coefficients of $tlmB^+-tl\floor{m(B^+-B)}+t(G+lR)$ are in a fixed finite rational set.
By assumption, the coefficients of $B$ are in a DCC set with rational accumulation points $\mathscr{R}$. Therefore the coefficients of $(1-tlm)B+tlmB^+-tl\floor{m(B^+-B)}+t(G+lR)$ are in a DCC set with rational accumulation points.

By Theorem \ref{FM}, there is a positive integer $n'$ depending only on $d$, $\mathscr{R}$ $m$, $l$ and $t$ such that there is an $n'$-complement $K_X+\Omega$ of $K_X+B+t(G+lR+l\{m(B^+-B)\})$, such that
\[\Omega\geq B+t(G+lR+l\{m(B^+-B)\}).\]

On the other hand, let
$$\Delta_{\overline{W}}\defeq B^+_{\overline{W}}+\frac{t}{m}A_{\overline{W}}-\frac{t}{lm}G_{\overline{W}}.$$
Then $(\overline{W},\Delta_{\overline{W}})$ is sub-$2\epsilon$-klt for some $\epsilon>0$ depending only on $\mathcal{P}$, $t$, $l$ and $m$ since $\mathrm{Supp}B^+_{\overline{W}}\subseteq\mathrm{Supp}\Sigma_{\overline{W}}\subseteq\mathrm{Supp}G_{\overline{W}}$, $(\overline{W}, \Sigma_{\overline{W}})$ is log smooth, $(\overline{W},B^+_{\overline{W}})$ is sub-lc, and $A_{\overline{W}}$ is not a component of $\ceil{B^+_{\overline{W}}}$. Moreover, $K_{\overline{W}}+\Delta_{\overline{W}}\sim_{\bQ} 0$.

Let
$$\Delta\defeq B^++\frac{t}{m}A-\frac{t}{lm}G.$$
Then $K_X+\Delta\sim_{\bQ} 0$. $(X,\Delta)$ is sub-klt since $K_X+\Delta$ is the crepant pullback of $K_{\overline{W}}+\Delta_{\overline{W}}$.

Let $\Theta=\frac{1}{2}\Delta+\frac{1}{2}\Omega$. Then
$$\Theta=\frac{1}{2}B^++\frac{t}{2m}A-\frac{t}{2lm}G+\frac{1}{2}\Omega$$
$$\geq\frac{1}{2}B^++\frac{t}{2m}A-\frac{t}{2lm}G+\frac{1}{2}B+\frac{t}{2}(G+lR)\geq-\frac{t}{2lm}G+\frac{t}{2}G\geq 0.$$

Since $(X,\Delta)$ is sub-$\epsilon$-klt, $K_X+\Delta\sim_{\bQ} 0$ and $(X,\Omega)$ is lc log Calabi--Yau, $(X,\Theta)$ is klt log Calabi--Yau. The coefficients of $\Theta$ belong to a fixed finite set depending only on $t$, $l$, $m$ and $n'$. Moreover, $\Theta\geq\frac{1}{2}B^++\frac{1}{2}\Omega\geq\frac{1}{2}B+\frac{1}{2}B=B$. 
Recall that $n(K_X+B^+)\sim 0$, $lA\sim G$ and $n'(K_X+\Omega)\sim 0$. Let $k$ be an integer such that $k\{\frac{1}{2n},\frac{t}{2lm},\frac{1}{2n'}\}\subseteq\mathbb{N}$, then $K_X+\Theta$ is a $k$-complement of $K_X+B$, with $(X,\Theta)$ $\epsilon$-lc and $\Theta\geq B$. It follows that $(X,\Theta)$ is $\frac{1}{k}$-lc
\end{proof}
\begin{thm}\label{DCCmain}
Fix a positive integer $d$, a positive real number $\theta$ and a DCC set $\mathscr{S}'$ of real numbers in $[0,1]$. The set $\mathcal{D}'$ of all klt weak Fano pairs $(X,B)$ satisfying
\begin{enumerate}
\item $\dim X=d$,
\item the coefficients of $B\in\mathscr{S}'$,
\item $\alpha(X,B)>\theta$, and
\item $(-(K_X+B))^d>\theta$
\end{enumerate}
forms a log bounded family.
Moreover, there is a rational number $k$, finite sets $S_1\subset [0,1]$ and $S_2\subset\bQ\cap[0,1]$, depending only on $d$, $\theta$ and $\mathscr{S}'$, such that for every element $(X,B)\in\mathcal{D}'$, there is a strong klt $(S_1,S_2,k)$-complement $K_X+\Theta$ of $K_X+B$.
\end{thm}
\begin{proof}
As in the proof of Theorem \ref{thm4.2}, by Theorem \ref{hx} and by Lemma \ref{qfact}, we may assume $X$ is $\bQ$-factorial and it is enough to show the existence of $k$.
Replacing $\mathscr{S}'$ by its closure, we may assume it is closed.

Let $\Gamma_1$, $\Gamma_2$, $r$, $a_i$, $g$, $\Gamma'_i$, $g_i$, $\mathscr{S}_i$ and $g'_i$ be as given in Theorem \ref{g'}.
Let $(X,B)\in\mathcal{D}'$ and write $B=\sum_j b_jB_j$, where $B_j$ are distinct Weil $\bQ$-Cartier divisors.
Then we have $B\leq\sum_i(a_i\sum_j g'_i(b_j)B_j)$ and $\sum_j g'_i(b_j)B_j\leq\frac{B+3\sum_j g(b_j)B_j}{4}$ for every $1\leq i\leq r$. Furthermore, $(X,\sum_j g(b_j)B_j)$ is lc and $-(K_X+\sum_j g(b_j)B_j)$ is pseudo-effective by Theorem \ref{ACCPELC}.
So $(X,\sum_j g'(b_j)B_j)$ is klt for every $1\leq i\leq r$.

For each fixed $i$, we consider the following argument. Sine $X$ is of Fano type and $-(K_X+\sum_j g'_i(b_j)B_j)$ is pseudo-effective, we may run a $-(K_X+\sum_j g'_i(b_j)B_j)$-MMP and reach a  model $Y_i$ such that $-(K_{Y_i}+\sum_j g'_i(b_j)B_{Y_ij})$ is nef, where $B_{Y_ij}$ is the strict transform of $B_j$ on $Y_i$. Therefore, by Lemma \ref{alpineq}, $({Y_i},\sum_j g'_i(b_j)B_{Y_ij})$ is a klt weak Fano pair.
Since $(X,\sum_j g(b_j)B_j)$ has an $\bR$-complement, so does $({Y_i},\sum_j g(b_j)B_{Y_ij})$.
On the other hand, by Lemma \ref{alpineq} and the fact that volume is invariant under small birational maps and is increased by pushing-forward through morphisms, we have
\begin{enumerate}[label=(\roman*)]
\item  $\mathrm{vol}({Y_i},-(K_{Y_i}+\sum_j b_jB_{Y_ij}))\geq\mathrm{vol}(X,-(K_X+B))>\theta$, and
\item  $\alpha({Y_i},\sum_j b_jB_{Y_ij})>\theta$.
\end{enumerate}

By Lemma \ref{qfact}, we may replace $X$ by $Y$ and replace $B_j$ by $B_{Y_ij}$ except that $(X,\sum_j g'_i(b_j)B_j)$ is klt weak Fano but $(X,B)$ might not be weak Fano any more.

We estimate that
\begin{align*}
&\mathrm{vol}(X,-(K_X+\sum_j g'_i(b_j)B_j))\\
=&\mathrm{vol}(-(K_X+\sum_j g(b_j)B_j)+\sum_j g(b_j)B_j-\sum_j g'_i(b_j)B_j)\\
\geq&\mathrm{vol}(-(K_X+\sum_j g(b_j)B_j)+\sum_j g(b_j)B_j-\frac{B+3\sum_j g(b_j)B_j}{4})\\
\geq&\mathrm{vol}(-\frac{1}{4}(K_X+\sum_j g(b_j)B_j)+\frac{\sum_j g(b_j)B_j-B}{4})\\
=&(\frac{1}{4})^d\mathrm{vol}(-(K_X+B)).
\end{align*}
On the other hand,
\begin{align*}
\alpha(X,\sum_j g'_i(b_j)B_j)=&\mathrm{lct}((X,\sum_j g'_i(b_j)B_j),|-(K_X+\sum_j g'_i(b_j)B_j)|_{\bR})\\
\geq&\mathrm{lct}((X,\frac{B+3\sum_j g(b_j)B_j}{4}),|-(K_X+B)|_{\bR})\\
\geq&\frac{1}{4}\mathrm{lct}((X,B),|-(K_X+B)|_{\bR})\\
=&\frac{1}{4}\alpha(X,B),
\end{align*}
where that last inequality follows from the following argument.
If $0<t\leq \frac{1}{4}\mathrm{lct}((X,B),|-(K_X+B)|_{\bR})$, then for any $M\in|-(K_X+B)|_{\bR}$, $(X,B+4tM)$ is lc.
It follows that $(X,\frac{B+4tM+3\sum_j g(b_j)B_j}{4})$ is lc by the linearity of log discrepancies because $(X,\sum_jg(b_j)B_j)$ is lc.
Therefore we have \[t\leq \mathrm{lct}((X,\frac{B+3\sum_j g(b_j)B_j}{4}),|-(K_X+B)|_{\bR}).\]

As a conclusion, $\mathrm{vol}(X,-(K_X+\sum_j g'_i(b_j)B_j))$ and $\alpha(X,\sum_j g'_i(b_j)B_j)$ are both bounded below away from $0$ and $(X,\sum_j g'_i(b_j)B_j)$ is a klt weak Fano pair.

By Theorem \ref{thm4.2}, for each $i$, there is a rational number $k_i$ depending only on $d$, $\theta$ and $\mathscr{S'}$ such that there is a strong klt $k_i$-complement $K_X+\Theta_i$ of $K_X+\sum_jg_i(b_j)B_j$.
Replacing $k_i$ by $k\defeq\prod_{h=1}^r k_h$ for each $i$, we may assume that $k_i=k$.
Let $S_1=\{a_i\}_i$ and $S_2=\frac{1}{k}\mathbb{N}\cap[0,1]$.
Let $\Theta\defeq\sum_{i=1}^r a_i\Theta_i$. Then $K_X+\Theta$ is the complement we want since $K_X+\Theta\geq K_X+\sum_i(a_i\sum_j g'_i(b_j)B_j)\geq K_X+B$.

\end{proof}

\begin{thm}\label{bddcomp}
Fix a positive integer $d$, and a positive real number $\theta$. Let $\mathcal{P}$ be a log bounded set of all klt Fano pairs $(X,B)$ satisfying
\begin{enumerate}
\item $\dim X=d$,
\item $\alpha(X,B)>\theta$, and
\item $(-(K_X+B))^d>\theta$.
\end{enumerate}
Then there is a natural number $k$, depending only on $d$, $\theta$ and $\mathcal{P}$, such that for every $(X,B)\in \mathcal{P}$, there is a strong klt $k$-complement $K_X+\Theta$ of $K_X+B$.
\end{thm}
\begin{proof} 
By Lemma \ref{qfact}, we may assume $X$ is $\bQ$-factorial.
Since $\mathcal{P}$ is log bounded and $(-(K_X+B))^d>\theta$, there is a positive real number $a$ such that for every $(X,B)\in \mathcal{P}$, $-(K_X+B+a\Supp B)$ is effective. Since $\alpha(X,B)>\theta$, by replacing $a$, we may assume that $(X,B+a\Supp B)$ is klt. Therefore, there is a natural number $k$ such that for every $(X,B)\in \mathcal{P}$, there is a divisor $\overline{B}\geq B$ on $X$ such that $\overline{B}\leq B+\frac{1}{2}a\Supp B$ and $k\overline{B}$ is an integral divisor. Therefore, $(X,2\overline{B}-B)$ is klt and $-K_X-2\overline{B}+B$ is pseudo-effective. Replacing $a$ further, we may assume that $(X,2\overline{B}-B)$ has an $\bR$-complement. We may run a $-(K_X+\overline{B})$-MMP and reach a model $Y$ such that $-(K_{Y}+\overline{B}_Y)$ is nef, where $\overline{B}_Y$ and $B_Y$ are the strict transforms of $\overline{B}$ and $B$ on $Y$ respectively.
Therefore, by Lemma \ref{alpineq}, $(Y,\overline{B}_Y)$ is a klt weak Fano pair, and $(Y,2\overline{B}_Y-B_Y)$ has an $\bR$-complement.
By Lemma \ref{alpineq} and the fact that volume is invariant under small birational maps and is increased by pushing-forward through morphisms, we have
\begin{enumerate}[label=(\roman*)]
\item  $\mathrm{vol}(Y,-(K_Y+B_Y))\geq\mathrm{vol}(X,-(K_X+B))>\theta$, and
\item  $\alpha(Y,B_Y)>\theta$.
\end{enumerate}

By Lemma \ref{qfact}, we may replace $X$, $B$ and $\overline{B}$ by $Y$, $B_Y$ and $\overline{B}_Y$ respectively except that $(X,\overline{B})$ is klt weak Fano but $(X,B)$ might not be weak Fano any more.
By the same computation as in Theorem \ref{DCCmain}, we can bound $\alpha(X,\overline{B})$ and $(-(K_X+\overline{B}))^d$ below away from $0$. Now we apply Theorem \ref{thm4.2} and we are done.
\end{proof}

\begin{cor}
Fix a positive integer $d$ and positive real numbers $\theta$ and $\delta$. Let $\mathcal{P}$ be the set of all projective pairs $(X,B)$ satisfying
\begin{enumerate}
\item $(X,B)$ is a klt weak Fano pair of dimension $d$,
\item $B$ is an effective $\mathbb{Q}$-divisor on $X$,
\item $\alpha(X,B)^d(-(K_X+B))^d>\theta$, and
\item coefficients of $B>\delta$.
\end{enumerate}
Then there is a natural number $k$, depending only on $d$, $\theta$ and $\delta$, such that for every $(X,B)\in \mathcal{P}$, there is a strong klt $k$-complement $K_X+\Theta$ of $K_X+B$.
\end{cor}
\begin{proof}
This follows from Throem \ref{LLX} and Theorem \ref{bddcomp}.
\end{proof}


\end{document}